\documentclass[reqno]{amsart}
  \input amssymb.sty

\usepackage{epsfig,latexsym,amsfonts,amssymb,amsmath,amscd,graphics,epic}
\usepackage{amsfonts,amssymb,amsmath,amscd,amsthm}
\usepackage[mathscr]{eucal}
\usepackage{mathrsfs}
\usepackage{oldgerm,units}
\usepackage{wrapfig,epsfig}

\usepackage{ifthen}
\usepackage{mathbbol}

\usepackage{amsthm}
\usepackage[mathscr]{eucal}
\usepackage{mathrsfs}
\usepackage{mathbbol}
\usepackage{oldgerm,units}
\usepackage{wrapfig}

\newtheorem{theorem}{Theorem}[section]

\newtheorem{definition}[theorem]{Definition}

\newtheorem{example}[theorem]{Example}
\newtheorem{remark}[theorem]{Remark}

\newcommand{\one}{\mathbb{1}}
\newcommand{\zero}{\mathbb{0}}

\newcommand{\trop}[1]{\mathcal{#1}}

\newcommand{\tG}{\trop{G}}
\newcommand{\tH}{\trop{H}}

\newcommand{\tS}{\trop{S}}
\newcommand{\tT}{\trop{T}}

\newcommand{\To}{\longrightarrow }

\newcommand{\al}{\alpha}
\newcommand{\bt}{\beta}

\newcommand{\lm}{\lambda}

\newcommand{\OP}{\left(}
\newcommand{\CP}{\right)}

    \ifx\proof\undefined
    \newenvironment{proof}{
    \smallskip
    \noindent\emph{Proof.}}{\hfill\(\Box\)
    \bigskip
    } \fi

\newcommand{\vvMat}[9]{\OP \begin{array}{ccc}
  #1 & #2 & #3\\
  #4 & #5 & #6\\
  #7 & #8 & #9\\
\end{array}\CP}

\newcommand{\ifdef}[3]{\ifthenelse{\equal{#1}{true}}{#2}{#3}}

\def\Skip{\vskip 1.5mm}
\def\pSkip{\vskip 1.5mm \noindent}

\def\hfA{\Inu{f_A}}
\def\tfA{\widetilde{f_A}}

\newcommand{\etype}[1]{\renewcommand{\labelenumi}{(#1{enumi})}}
\def\eroman{\etype{\roman}}

\newcommand{\Det}[1]{ \left|{#1}\right|}

\def\essn{{\operatorname{es}}}

\def\Bl{B_\ell}

\def\nucong{\cong_\nu}
\def\nug{>_\nu}
\def\nul{<_\nu}
\def\nuge{\ge_\nu}
\def\nule{\le_\nu}

\def\pipeGS{{\underset{\operatorname{\, gs }}{\mid}}}

\def\lmod{\mathrel   \pipeGS \joinrel\joinrel \joinrel =}
\def\lmodh{\lmod}

\def\lmodg{\lmod}

\def\tGz{\mathcal G_\zero}
\def\tTz{\mathcal T_\zero}
\def\tHz{\mathcal H_\zero}

\def\tH{\mathcal H}

\def\FunR{\operatorname{Fun} (R^{(n)},R)}
\def\CFunR{\operatorname{CFun} (R^{(n)},R)}

\def\({\left(}
\def\){\right)}

\def\multi{multicycle}
\def\nmulti{$n$-multicycle}

\newcommand{\tGinf}{\tG_{\zero}}

\def\hnu{\hat \nu}
\newcommand{\inua}[1]{\widehat{#1}}
\newcommand{\inu}[1]{\hat{#1}}
\newcommand{\Inu}[1]{\widehat{#1}}
\newcommand{\Bnu}[1]{\overline{#1}}

\def\RGnu{(R,\tGz,\nu)}

\def\cyc{C}

\def\quasi{quasi}

\def\ghost{\text{ghost}}

\def\a{\alpha}
\def\la{\lambda}
\def\sig{\sigma}

\def\one{\mathbb{1}}
\def\zero{\mathbb{0}}

\def\regular{nonsingular}

\def\um{I}

\def\nb{\nabla}
\def\nbnb{{\nb \nb}}
\def\Bnb{{\Bnu{\nabla}}}
\def\Bnbnb{{\Bnb \Bnb}}
\def\Tnb{{\Inu{\nabla}}}

\def\ra{a}
\def\bfa{\textbf{\ra}}

\newcommand{\adj}[1]{\operatorname{adj}(#1)}
\newcommand{\Tadj}[1]{{\Inu{\operatorname{adj}(#1)}}}

\def\rone{\one_R}
\def\rzero{\zero_R}

\def\fzero{\zero_F}

\def\la{\lambda}

\newtheorem{thm}[theorem]{Theorem}
\newtheorem*{thm*}{Theorem}
\newtheorem*{dig*}{Digression}

\newtheorem{cor}[theorem]{Corollary}

\newtheorem{lem}[theorem]{Lemma}
\newtheorem{rem}[theorem]{Remark}
\newtheorem{prop*}{Proposition}

\newtheorem{prop}[theorem]{Proposition}
\newtheorem{defn}[theorem]{Definition}
\newtheorem*{examp*}{Example}
\newtheorem*{examples*}{Examples}
\newtheorem*{remark*}{Remark}

\newtheorem*{defn*}{Definition}
\newtheorem*{note*}{Note}

\begin{document}
\title[Supertropical Matrix Algebra II: Solving tropical equations] {Supertropical Matrix Algebra
II: \\ \vskip 1mm Solving tropical equations}

\author[Zur Izhakian]{Zur Izhakian}\thanks{The first author has been supported by the
Chateaubriand scientific post-doctorate fellowships, Ministry of
Science, French Government, 2007-2008}
\thanks{This research is supported in part by the
Israel Science Foundation, grant No. 448/09.}
\address{Department of Mathematics, Bar-Ilan University, Ramat-Gan 52900,
Israel}
\email{zzur@math.biu.ac.il}
\author[Louis Rowen]{Louis Rowen}
\address{Department of Mathematics, Bar-Ilan University, Ramat-Gan 52900,
Israel} \email{rowen@macs.biu.ac.il}

\thanks{\textbf{Acknowledgement:} The authors would like to thank the
referee for his helpful comments, especially in asking whether the function $\hat \nu$
can be taken to be multiplicative.}

\subjclass[2000]{Primary 15A09, 15A03, 15A15, 65F15; Secondary
16Y60 }

\date{\today}

\keywords{Supertropical matrix algebra, Adjoint matrix, Cramer's
rule, Solving linear equations, Supertropical eigenvectors.}


\begin{abstract}
We continue the study of matrices over a supertropical algebra,
proving the existence of a tangible adjoint of $A$, which provides
the unique right (resp.~left) quasi-inverse   maximal with respect
to the right (resp.~left) quasi-identity matrix corresponding to
$A$; this provides a unique maximal (tangible) solution to
supertropical vector equations, via a version of Cramer's rule. We
also describe various properties of this tangible adjoint, and use
it to compute supertropical eigenvectors, thereby producing an
example in which an $n\times n$ matrix has $n$ distinct
supertropical eigenvalues but their supertropical eigenvectors are
tropically dependent.

\end{abstract}

\maketitle




\section{Introduction}
\numberwithin{equation}{section}

This paper is a continuation of \cite{IzhakianRowen2008Matrices};
here, we solve vector equations in supertropical algebra,
using the tangible version $\adj{A}$
of the adjoint, which yields a
version of Cramer's rule (Theorem \ref{tangsol} below). This
solution is the unique maximal solution in a certain sense (Theorem \ref{tang3}). In \S
4 we compare $\adj{\adj{A}}$ to $A$. These computational techniques using the adjoint are quite powerful; in Theorem~\ref{sing12},
we apply them to compute supertropical eigenvectors and to refute the natural conjecture that the supertropical eigenvectors would be tropically independent when their supertropical eigenvalues are distinct.

Some of the parallels to classical matrix theory are quite
unexpected, especially since their natural analogs in the max-plus
algebra often fail. See \cite{ABG} for some of the max-plus
theory; related references are given in the bibliography of
\cite{IzhakianRowen2008Matrices}.  However, the supertropical
algebra, which covers the max-plus algebra, is endowed with the ``ghost surpassing''
relation $\lmodg$ given in Definition \ref{lmod}, which
specializes to equality on the ``tangible'' elements, and provides
suitable analogs of these basic results from matrix theory.

The paper \cite{AGG} was written independently of the earlier
version of this paper, and contains some relevant results,
especially an elegant meta-theorem about identities of matrix
semirings described in Section~\ref{matid} below. During the
course of the current version of this paper, we indicate how the
results of~\cite{AGG} interact with our results.

We recall that this work is in the environment of a
 \textbf{semiring with ghosts} \cite{IzhakianRowen2007SuperTropical}, which is a
triple $\RGnu,$ where $R$ is a semiring with zero element,
$\zero_R$, (often identified in the examples with $-\infty$, as
indicated below), and $\tGinf = \tG \cup \{ \zero_R \} $ is a
semiring ideal, called the \textbf{ghost ideal}, together with an
idempotent semiring  homomorphism
$$\nu : R \ \To \ \tG\cup \{\zero_R\}$$  called the \textbf{ghost map}, i.e.,
which preserves multiplication as well as addition, defined as
 $$\nu (a) = a+a
 .$$

 We write $a^{\nu }$ for $\nu(a)$, called
the $\nu$-\textbf{value}  of $a$. We write $a \nuge  b$, and say
that $a$ \textbf{dominates} $b$, if $a^\nu \ge b^\nu$. Likewise we say
that $a$   \textbf{stricly dominates} $b$, written $a >_ \nu  b$, if $a^\nu >  b^\nu$. Two
elements $a$ and $b$ in $R$ are said to be
$\nu$-\textbf{matched} if they have the same $\nu$-value, in
which case we also
write $a \nucong b$.

\subsection{Supertropical semirings}

\begin{defn}
A \textbf{supertropical semiring} is a semiring with ghosts that
has the extra properties:
\begin{enumerate} \eroman
 \item  $a+b   =  a^{\nu } \quad \text{if}\quad a^{\nu } =
 b^{\nu}$; \pSkip
 \item $a+b  \in \{a,b\},\ \forall a,b \in R \;s.t. \; a^{\nu }
\ne b^{\nu }.$ (Equivalently, $\tGinf$ is ordered, via $a^\nu \le
b^\nu$ iff $a^\nu + b^\nu = b^\nu$.) \pSkip

\end{enumerate}

  A \textbf{supertropical domain} (the focus of interest for us) is a commutative supertropical
semiring $R = \RGnu$ in which the following two extra conditions are
satisfied:
\begin{enumerate} \eroman
    \item  $R\setminus \tGinf$  is a monoid~$\tT $ with respect to the semiring multiplication;
the elements of $\tT$ are called  \textbf{tangible}.  \Skip
       \item The map $\nu _\tT : \tT \to \tG$ (defined as the restriction from $\nu$ to $\tT$) is onto; in other words, every element of $\tG$ has the form
$a^\nu$ for some $a\in \tT$.  \Skip
   \end{enumerate} \end{defn}

    We write $\tTz$ for $\tT \cup \{ \rzero\}$. Note that  $\tTz$  acts as the max-plus algebra, except in the
case when $a^\nu = b^\nu$, in which case the ghost layer plays its
role.

\begin{defn}
A \textbf{supertropical semifield} is a supertropical domain
$\RGnu$ in which every  tangible element is invertible; in other
words, $\tT$ is a multiplicative group. Thus, $\tG$ is also a multiplicative group.
\end{defn}

Recall from \cite[Remark 3.12]{IzhakianRowen2007SuperTropical}
that any supertropical domain $R$ is $\nu$-\textbf{cancellative},
in the sense that $c a^\nu = c b^\nu$ for $c \ne \rzero$ implies
$a^\nu = b^\nu,$ and in particular its ghost ideal $\tG$ is
cancellative as a multiplicative monoid.
 Since any commutative cancellative monoid has an
Abelian group of fractions, one often can reduce from the case of
a supertropical domain to that of a supertropical semifield. (More
details are given in \cite[Proposition~3.19 and Remark~3.20]
{IzhakianRowen2007SuperTropical}.)

\subsection{The supertropical relation ``ghost surpasses''}

The following  relation, stronger than $\nuge$, plays a key role
in the theory, and especially in this paper.

\begin{definition}\label{lmod} We say $b = a + \ghost $ if $b = a + c$ for $c$ some   ghost
element.
 We define the relation $\lmodg$, called ``ghost surpasses,'' on any
semiring with ghosts $R$, by $$b \lmodg a \qquad \text{ iff }
\qquad b = a +
  \ghost.$$
\end{definition}

Note that $b \lmodg
 a$  implies $a + b \in \tGz.$
In a supertropical semiring, $b \lmodg a$ iff $b = a$ or $b$ is a
ghost $\nuge a.$ In particular, if $b \lmodg  a$, then $b \nuge
a$. (The converse is false, since one could have tangible $b \nug
a$.) In fact  the relation $\lmodg$ is a partial order on $R$, and
is not symmetric; for example $a^\nu \lmodg a$, for $a$ tangible,
but not visa versa; i.e., $a \not\lmodg a^ \nu$.

\begin{rem}\label{antitan} In a supertropical domain, if $a$ is tangible with $a \lmodg
 b$ , then $a= b.$  (Indeed, write $a = b+c$ with $c\in \tGz.$ Then $b \not \nucong c$
 since $a \in \tT,$ and likewise $a\ne c,$ since $a$ is tangible,  so $a=b$.) \end{rem}

   Thus, for tangible elements, the relation $\lmodg$ generalizes equality in the
max-plus algebra, and seems to be the ``correct'' generalization
  to enable us to find analogs of theorems from classical
linear algebra. This is the reason for our use of the symbol
$\lmodg$, not to be
confused with the usage in model theory. On the other hand, we
have the following observation.

\begin{lem} The relation  $\lmodg$ is antisymmetric in any supertropical semiring.
\end{lem}
\begin{proof} We need to show that if $a \lmodg  b$ and $b \lmodg a$, then $a = b.$ This holds by Remark~\ref{antitan} if $a$ is tangible (and thus,
by symmetry, if $b$ is tangible). Hence, we may assume that $a,b
\in \tGz,$ in which case $$a = a^\nu = b^\nu = b.$$\end{proof}

\subsection{The tangible retract function}

Although in general, the map $\nu_\tT : \tT \to \tG$ need not be
1:1 in a supertropical domain, $\nu_\tT $ is onto by definition;
we find it convenient to choose a ``tangible retract'' function
$\hnu : R \to \tTz$ restricting to the identity map on $\tTz$,
such that
 $\nu \circ \hnu$ restricts to the identity
map on $\tGz$. We write $\inu{ b}$ for~$\inu \nu(b)$; thus,
$(\inu{ b})^\nu = b$ for all $b\in \tGz$. We retain the notation
$\hnu$ when working with more complicated expressions.

We do not see any general way to define $\hnu$  on $\tG$ other
than applying the axiom of choice rather freely, although in
special cases there are canonical definitions for $\hnu$ (such as
when $\nu$ is 1:1 or a ``lowest term'' valuation on power series).

 \begin{prop}\label{multip}  If $F$ is a divisibly closed semifield, the map
 $\hnu : F \to \tT$ can be defined such that its restriction
 $\hnu _\tG: \tG \to \tT$ is a multiplicative group homomorphism.\end{prop}
 \begin{proof}  Consider all pairs $(M, \hnu)$, where $M \subset (\tG,\cdot)$
 is a subgroup with a partial tangible retract function $\hnu_M: M \to \tTz$ that is multiplicative.
 We order these pairs by saying $$(M, \hnu_M)> (M', \hnu_{M'}) \quad \text{if} \quad M \supset
 M'$$
 and $ \hnu_M$ restricts to $\hnu_{M'}$ on $M'$, i.e., ${ \hnu_M |} _{M'} = \hnu_{M'}$.

 By Zorn's lemma, there is a subgroup $M \subset (\tG,\cdot)$ for which $(M,\hnu_M)$ is maximal.
  If $M \ne \tG,$ then take $a\in \tG \setminus M.$  Let $$P =  \{ n\in \mathbb Z: a^n \in
  M\},$$
  an  ideal of $\mathbb Z,$ and write $P = k\mathbb Z$  for some    $k \ge 0$.
   If $k>0$, choose $\hat a$ such that $\hat a^k = \Inu{a^k}$.
(This is possible since $F$ is divisibly closed.)
 If $k = 0$, choose
 $\hat a$ arbitrarily in~$\tT$ such that $(\inu{ a})^\nu = a$.
 Define $$\Inu{a^ib} := \Inu{a^i} \hat b$$ for each $b\in M$ and each $i>0$. To see that this
 is well-defined, suppose $a^i b = a^j b'$ for $i \ge j$ and $b' \in M$. Then $a^{i-j} =  b'  b^{-1}\in M,$ which by
 definition is $\Inu{a^{i-j}} = \hat a^{i-j}$ since $k$ divides $i-j$,
  implying  $\hat a^{i}\hat b = \hat a^{j} \widehat{b'}.$ This contradicts the maximality
  of $M$, so we must have $M = \tG.$ Then we put $\Inu{\fzero} = \fzero.$
 \end{proof}

\begin{rem}   Whenever $a \not \nucong b$, the retract map   $\hnu$ must satisfy
 \begin{equation}\label{hatadd} \Inu{a+b} =  \hat a + \hat b. \end{equation}  Indeed, we may assume that $a >_\nu b,$ and thus
 $\Inu{a+b} = \hat a = \hat a + \hat b  .$

 But for  $a \in \tG $, we have $\hat a + \hat a = (\hat a)^\nu = a$ which is not $ \hat a = \widehat{a+a},$
 so $\hnu$   is not a semiring homomorphism.
\end{rem}

 The
following observation enables us to utilize   $\hnu$ to make
calculations  paralleling those in the max-plus algebra.

\begin{prop}\label{tang} If $ \sum _k a_k \Inu{b_{j,k}}  \in \tGz$
for each $1 \le j\le m$, then  $ \sum _{k} {a_k}\hnu(\sum _{j=1}^m
{b_{j,k}c_j})\in \tGz$ for any $c_j \in R$.
\end{prop}
\begin{proof}  Otherwise, consider the single dominating term ${a_{k_1}}
\Inu{({b_{{j_1},{k_1}}c_{j_1}})}$ of the right side.  We are done
unless $a_{k_1} \in \tT.$  But  ${a_{k_1}}
\Inu{({b_{{j_1},{k_1}}c_{j_1}})}$ dominates ${a_{{k}}}
\Inu{({b_{{j_1},{k}}c_{j_1}})}$ for each ${k}$, implying
${a_{k_1}} \Inu{b_{{j_1},{k_1}}}\in \tT$ dominates each $
a_{{k}} \Inu{b_{{j_1},{k}}}$. Thus, there must be~${k_2}$ with
${a_{k_1}} \Inu{b_{{j_1},{k_1}}}\nucong {a_{{k_2}}}
\Inu{b_{{j_1},{k_2}}}$. But then $${a_{k_1}}
\Inu{({b_{{j_1},{k_1}}c_{j_1}})}\nucong {a_{{k_2}}}\Inu{
{(b_{{j_1},{k_2}}c_{j_1}})}$$ implying that their sum is ghost.
\end{proof}

\begin{prop}\label{tang0}  $\sum _k a_k \sum _j \Inu{b_{j,k}} \lmodg  \sum _k a_k  \Inu{(\sum _j
b_{j,k})}
 .$
\end{prop}
\begin{proof} The two sides are $\nu$-matched, so it remains to
show that if the left side is tangible, then it equals the right
side.

Suppose that $a_k \Inu{b_{j',k}}$ alone dominates the left side.
Then $  b_{j,k} <_\nu  b_{j',k}$ for each $j \ne j',$ implying
$\sum _j b_{j,k} = b_{j',k}.$ Hence,
  the
single dominating term in the summation at the right must also be
$a_k \Inu{ b_{j',k}}$.
\end{proof}

\subsection{Vectors}

 We also recall
the definition of $R^{(n)}$ as the Cartesian product $\prod
_{i=1}^n R$ of $n$ copies
     of the  supertropical semiring $R$, viewed as a module
     via   componentwise multiplication,
    with zero element $\zero = ( \rzero)$
    and ghost submodule $\tHz = \tGz^{(n)}$. Let $\tH = \tHz \setminus \{\rzero\}.$
    When $R$ is a supertropical semifield, $R^{(n)}$~is called a
\textbf{tropical vector space} over $R$. A vector $\ne \zero$ is
called \textbf{tangible} if all of its components are in~$\tTz.$

Our partial orders $\lmodg$ and $\nuge$ on $R$, and the tangible
retract function $\hnu: R \to \tTz $, extend respectively to the
partial orders $\lmodh$ and $\nuge$
 on $R^{(n)}$, and the tangible
retract function $\hnu : R^{(n)} \to \tTz^{(n)}$, by matching the
corresponding
 components;
 note that
vectors $v,w$ satisfy $w \lmodh  v$ iff $w = v+\ghost.$ For example,
$$  (\rone^\nu,\rone^\nu,\rone^\nu )  \  \lmodh  \
(\rone^\nu,\rone^\nu,\rone) \ \lmodh  (\rone^\nu,\rone,\rone ) .$$
Also, by checking components, we see
 that $\lmodh $  is   antisymmetric for vectors.

\begin{lem} Suppose $v,w \in R^{(n)}$, with $w$ tangible.
Then $v \lmodh  w$ iff $v+w \in \tHz$.
\end{lem}\begin{proof} $(\Rightarrow)$ is obvious.

$(\Leftarrow)$ By assumption,  each component $w_i$ of $w$ is in
$\tTz$, and $v_i+w_i \in \tHz$ implies $v_i = w_i$ or $v_i$ is
ghost $\nuge w_i;$ thus $v_i \lmodg w_i$ for each $i$, implying $v
\lmodh  w.$
\end{proof}

\subsection{The $\nu$-topology}

We also need the following topology on $R$;~cf.~\cite[Definition
3.22]{IzhakianRowen2007SuperTropical}:

\begin{defn}\label{ordertop}  Suppose $\RGnu$ is a supertropical domain.
 Viewing $\tG$ as an ordered monoid with respect to $\nuge$,
we define the $\nu$-\textbf{topology} on $R$, whose open sets have
a base comprised of the \textbf{open intervals}
$$W _{\alpha, \beta} = \{ a \in R: \alpha < a^\nu <
\beta\}; \qquad W _{\a, \beta; \tT} = \{ a \in \tT: \a < a^\nu <
\beta\}, \quad \alpha, \beta \in \tG_\zero.$$ This topology
extends to the product topology on $R^{(n)}$ for any $n$.
\end{defn}
\noindent  Note that the tangible vectors in $R^{(n)}$ are a dense
subset in
 the $\nu$-topology.
 When we need to apply topological arguments, in order that multiplication be a continuous function, we assume
 that $\tT$ is dense, in the sense that $W _{\a, \beta; \tT}\neq \emptyset$ whenever $\a \nul
 \beta$.

\subsection{The semiring of functions}

Let $\FunR$ (resp.~$\CFunR$) denote the semiring of functions
(resp. continuous functions) from $R^{(n)}$ to~ $R^{(n)}$;
~cf.~\cite[Definition 3.31]{IzhakianRowen2007SuperTropical}. We
can also define our partial orders on $\FunR$:

\begin{defn} For $f,g \in \FunR,$ we write
 $f \nuge g$ if $f(\bfa)\nuge g(\bfa)$  for all $\bfa = (a_1, \dots, a_n)$ in~ $R^{(n)}.$

 The \textbf{ghost-surpassing identity} $f \lmodg g$
holds  for $f,g \in
\FunR$, if $f(a_1, \dots, a_n) \lmodg g(a_1, \dots, a_n)$ for
every $a_1, \dots, a_n \in R.$
  \end{defn}

\begin{prop}\label{dens}  Suppose $f,g \in \CFunR$.
\begin{enumerate} \eroman
    \item  If  $f(\bfa) \nuge g(\bfa)$ for all $\bfa$ in a dense
    subset
of $R^{(n)}$, then $f \nuge g$. \pSkip

    \item If  $f(\bfa) \lmodg  g(\bfa)$ for all $\bfa$ in a dense
subset $\mathcal S$ of $R^{(n)}$, then $f \lmodg g$.
\end{enumerate}

\end{prop}
\begin{proof}  $ $
\begin{enumerate}\eroman
    \item  Otherwise, we have $f(\bfa) \nul g(\bfa)$
for some $\bfa \in R^{(n)}$, so this inequality holds for some
open interval $W_\bfa$ containing $\bfa.$ \pSkip

\item  We are done by (i) unless there exists $\bfa$ such that
$f(\bfa) \in \tT$ and $f(\bfa) \nug   g(\bfa)$. But then this
inequality holds for some open interval $W_\bfa$ containing
$\bfa,$  implying $f( \bfa ') \in \tGz$ for all $\bfa' \in \tS
\cap W_\bfa$. We conclude that $f(W_\bfa )\subseteq \tG,$ contrary
to $\bfa \in W_\bfa.$
\end{enumerate}
\end{proof}

Several examples of ghost surpassing identities given in
\cite{IzhakianRowen2008Matrices}; as we shall see, many of these
can be obtained via a powerful new technique of \cite{AGG}.

\subsection{Identities of  semirings with symmetry}\label{symm}

Any commutative semiring with ghosts is a semiring with symmetry
in the sense of \cite[Definition 4.1]{AGG}, where their map $\tau$
is taken to be the identity map, and their $S^o$ is the ghost
ideal $\tGz$. Furthermore, they define a  relation $a \succeq^o b$
when $a = b + c$ for some $c \in S^o$;  this clearly specializes
to our relation $a \lmodg b.$

Akian, Gaubert, and Guterman \cite[Theorem~4.21]{AGG} then proved
their \textbf{strong transfer principle}, which we rephrase slightly:

\begin{thm}\label{STP} Suppose $ p^+ ,p^-, q^+ , q^- \in \mathbb N[\xi_1, \dots, \xi_m]$ are
polynomials in commuting indeterminates $\xi_1, \dots, \xi_m$, and
let $p =  p^+ - p^-$ and $q =  q^+ - q^-$ in the free commutative
ring $\mathbb Z[\xi_1, \dots, \xi_m]$. If $p = q$, and if no
monomials appear in both $q^+$ and $q^-$, then $p ^+ + p^-\succeq^o q^+ + q^-$ is
an identity for all commutative semirings.
\end{thm}

\section{Matrices and adjoints}

In this section, we accumulate basic information about matrices
and their adjoints. We write $M_n(R)$ for the semiring of $n
\times n$ matrices, whose multiplicative identity is denoted as
$I$, and we define the supertropical \textbf{determinant} $|A|$ of
$A = (a_{i,j})$ to be the permanent as in
\cite{zur05TropicalAlgebra,zur05TropicalRank,IzhakianRowen2008Matrices,
IzhakianRowen2009TropicalRank}; i.e.,
$$ |A | = \sum_{\sig \in S_n} a_{1, \sig(1)} \cdots a_{n, \sig(n)}.$$
 A permutation $\sig \in S_n$
 \textbf{attains} $\Det{A}$ if $\Det{A} \nucong  a_{\sig (1),1} \cdots a_{\sig
(n),n},$ where $A = (a_{i,j})$.

A matrix $A$ is defined to be \textbf{nonsingular} if $\Det{A}\in
\tT$ is invertible; $A$ is defined to be \textbf{singular} if
$\Det{A}\in \tGz$. Thus, over a supertropical semifield, every
matrix is either singular or nonsingular.

\begin{defn} The \textbf{minor} $A_{i,j}$ is obtained by
deleting the $i$ row and $j$ column of $A$. The \textbf{adjoint}
matrix $\adj{A}$ of a matrix $A =(a_{i,j})$ is defined as the
transpose of the matrix $(a'_{i,j}),$ where $a'_{i,j}=
\Det{A_{i,j}}$. The \textbf{tangible adjoint} matrix $\Tadj{A}$
of  $A$ is defined as the transpose of the matrix
$(\Inu{a'_{i,j}}).$

\end{defn}

Note that $\Tadj{A}$ depends on the choice of the tangible retract
function $\hnu$.

Viewing matrices as $n^2$-dimensional vectors, we can introduce
the product topology, as well as our relations $\nuge$ and
$\lmodg$, to matrices (by comparing the corresponding entries).

\begin{lem} The function $\det : A \mapsto \Det{A}$ is a continuous function from $ R^{(n^2)}$ to
$R$, and the adjoint is a continuous function from $ R^{(n^2)}$ to
$ R^{(n^2)}$.\end{lem}
\begin{proof}  Clear, because the determinant is defined in terms of addition and multiplication, which are
continuous functions in the $\nu$-topology over $R^{(n^2)}$.
\end{proof}

\begin{rem}\label{ord0} We can reformulate \cite[Theorem 3.5]{IzhakianRowen2008Matrices} as
$$\Det{AB} \lmodg \Det{A}\Det{B},$$ for any $A,B \in M_n(R)$,  and \cite[Proposition
4.8]{IzhakianRowen2008Matrices} as $\adj{AB} \lmodg \adj{B}\adj{A}
.$\end{rem}

\subsection{Ghost-surpassing identities of matrices}\label{matid}

Suppose   $P^+ =(p_{i,j}^+), P^- =(p_{i,j}^-), Q^+= (q _{i,j}^+)$,
and $ Q^- = (q_{i,j}^-)$ are matrix expressions whose respective
$(i,j)$   entries
 $p_{i,j}^+,p_{i,j}^-, q_{i,j}^+,$ and $q_{i,j}^- \in \mathbb N[\xi_1, \dots, \xi_m]$ are semiring  polynomials in the
entries of $x_1, \dots, x_\ell$ (in other words, only involving
addition and multiplication, but not negation). In particular,
when $x_i$ are $n \times n$ matrices, we set $m  = \ell n$.

Formally set $P(x_1, \dots, x_\ell)= P^+ - P^-$ and $ Q(x_1,
\dots, x_\ell)= Q^+ - Q^-$. Let us spell out  how
Theorem~\ref{STP} works for matrices.
 We say $Q$ is \textbf{admissible} if
the monomials of $q_{i,j}^+$ and $q_{i,j}^-$ are distinct, for each pair $(i,j)$.

Theorem~\ref{STP} provides the following metatheorem:

\begin{thm}\label{STP1} Suppose $P=Q$ is a matrix identity of $M_n(\mathbb Z)$, with $Q$ admissible.
(In other words, $P(A_1, \dots, A_\ell) = Q(A_1, \dots, A_\ell) $
for all matrices $A_1, \dots, A_\ell.$) Then for any commutative
semiring with ghosts $(R, \tGz, \nu)$, the matrix semiring with
ghosts $M_n(R)$ satisfies the ghost-surpassing matrix identity $P ^+ + P^-\lmodg Q^+ + Q^-
.$
\end{thm}

The proof is standard: It is enough to check for substitutions to
``generic matrices'' in which each indeterminate $x_k$ is
specialized to a matrix $(\xi_{i,j}^{k})$ whose entries are
commuting indeterminates. Then the proposed ghost-surpassing
identity $P ^+ + P^-\lmodg Q^+ + Q^-
$ can be expressed in terms of $n^2$ ghost-surpassing identities
in the commuting indeterminates $\xi_{i,j}^k,$ one for each matrix
entry.

\begin{rem}\label{background} Define the \textbf{characteristic
polynomial} $f_A$ of $A$ as $|A + \lm I|$. If    $ f_A =
\sum_{i=0}^n  \a_i \la^i, $ define  the \textbf{tangible
characteristic polynomial} $\hfA$ of $A$ as $ \sum_{i=0}^n
\Inu{\a_i} \la^i$. Here are some results from
\cite{IzhakianRowen2008Matrices}, which are reproved as easy
applications of Theorem~\ref{STP1}, for any semiring with ghosts
$(R, \tGz, \nu)$:

\begin{enumerate}\eroman
\item $\Det{AB} \lmodg \Det{A}\Det{B};$
 \item Any matrix $A$   satisfies its tangible characteristic polynomial   $\hfA  $;
 i.e.,
 $\hfA(A) \lmodg \zero;$
\item Notation as above, let $\tfA = \sum_{i=1}^n \a_i
\la^{i-1};$ then $\tfA(A) \lmodg \adj{A}.$
\end{enumerate}
In order to apply Theorem~\ref{STP1}, one needs to observe
that in each of these expressions  the  $q_{i,j}^+$ and $q_{i,j}^-$
are distinct. This is true in (i) and (iii) because of the standard formulas for the determinant and adjoint.
(We can describe   the left side  of  (iii) by applying Newton's formula for computing the coefficients of the  characteristic polynomial of $A$ in terms of traces of powers of $A$. )

Now (ii) is  obtained by multiplying (iii) by $A$, noting that $A
\adj{A}\lmodg I$ by \cite[Remark~4.14]{IzhakianRowen2008Matrices}.

\end{rem}

\subsection{Quasi-identity matrices and quasi-inverses}

Recall the following definition from
\cite{IzhakianRowen2008Matrices}:
\begin{defn}
 A \textbf{\quasi-identity matrix} is a \regular ,
multiplicatively idempotent matrix equal to the identity matrix
~$\um$ on the diagonal, and whose off-diagonal entries are in
$\tGz$.\end{defn}


\begin{rem}\label{eq:idorder} Any \quasi-identity matrix $\um'$  ghost surpasses $
\um;$ i.e., $I' \lmodg I$.
\end{rem}

Quasi-identities seem to be the key to the supertropical matrix
theory. Note however that the sum of quasi-identities is not necessarily a
quasi-identity. For example, take
 $$B_1 =  \left(\begin{matrix} 0 & 10^\nu \\
-\infty  & 0
\end{matrix}\right) \qquad \text{and} \quad B_2 =  \left(\begin{matrix} 0 & -\infty \\
10 ^\nu & 0
\end{matrix}\right).$$ Then $ B_1$ and $  B_2$ are
quasi-identities, but $$   B_1+ B_2 =  \left(\begin{matrix} 0 & 10^\nu \\
10^\nu & 0
\end{matrix}\right)$$ is a singular matrix Accordingly, we start with a given
 matrix $A$. Most of the following theorem is
contained in \cite[Theorem 4.12]{IzhakianRowen2008Matrices}.
\begin{thm}\label{rmk:adj}
 Suppose $A = (a_{i,j}),$ with $\Det{A}$ invertible. Define
 $$A^\nb = \frac \rone {\Det{A}} \adj{A} \quad  \text{and} \quad  A^{\Tnb} = \frac \rone {\Det{A}}
 \Tadj{A};$$
$$ \um_A = A A^\nb   = A A^{\Tnb}; \qquad \um'_A = A^\nb  A = A^{\Tnb}
A.
$$
   Then $A A^\nb   = A A^{\Tnb} = \um_A$ and
$A^\nb  A = A^{\Tnb}  A =\um'_A$ are quasi-identities.\end{thm}
\begin{proof} Starting with \cite[Theorem
4.12]{IzhakianRowen2008Matrices}, it remains to show that $A A^\nb
= A A^{\Tnb}$. Their $\nu$-values are the same, so we need only
check that the diagonal entries of $A A^{\Tnb}$ are tangible
(which is a fortiori, since this is true for the diagonal entries
of $A A^\nb$), and  that the off-diagonal entries of  $A A^{\Tnb}$
are ghost, which holds because of \cite[Remark
4.5]{IzhakianRowen2008Matrices}.\end{proof}

The fact that $I_A^2 = I_A,$ proved in \cite[Theorem
4.12]{IzhakianRowen2008Matrices} by means of Hall's Marriage
Theorem, is a key ingredient of the theory.

Inspired by Theorem \ref{rmk:adj}, when $\Det{A}$ is invertible,
we say that $A$ is
\textbf{\quasi-invertible} and call $A^\nb$ the \textbf{canonical two-sided \quasi-inverse} of
$A$ and  define the \textbf{ right  \quasi-identity matrix of} $A$
to be the matrix
$$I_A = A A^\nb   = A A^{\Tnb},$$ and the \textbf{left
\quasi-identity matrix of} $A$ to be the matrix $$I'_A =  A^\nb A
= A^{\Tnb} A.$$

(The tangible \quasi-inverse $A^{\Tnb}$ is introduced here since it plays a role in
solving equations, in \S\ref{solveq}.) Over a supertropical semifield, a matrix is
\quasi-invertible iff it is nonsingular.

\begin{rem}\label{lem:porder} If $ C \lmodg A,$ then  $BC \lmodg
BA.$ In particular, $B I' \lmodg  B$ for any quasi-identity matrix
$I'$; cf. ~Remark~\ref{eq:idorder}. By symmetry, we also have $I'
B \lmodg  B$.
\end{rem}

\begin{rem}\label{match1} If $A \nucong B$ are \quasi-invertible, then $I_A = I_B.$ (Indeed, the
diagonal of each is  the identity matrix $I$, and the off-diagonal
entries are clearly $\nu$-matched and thus, being ghosts, are
equal.)\end{rem}

\begin{example}\label{twobytwo} For $A  =  \left(\begin{matrix} a & b \\
c & d
\end{matrix}\right),$ we have $ \adj{A} = \left(\begin{matrix} d & b \\
c & a
\end{matrix}\right);$ hence
 $$ A\adj{A} = \left(\begin{matrix} |A| & (ab)^\nu \\  (cd)^\nu &
 |A|\end{matrix}\right) \quad \text{ whereas } \quad \adj{A}A = \left(\begin{matrix} |A| & (bd)^\nu \\  (ac)^\nu &
 |A|\end{matrix}\right).$$
Thus, the left and right quasi-identities of a quasi-invertible
matrix can be quite different. This enigma will only be resolved
in Corollary~\ref{switch} below.
 \end{example}

\begin{lem}\label{ord1} The  \quasi-invertible matrices are dense in $M_n(R)$.\end{lem}
\begin{proof}  Given any matrix, we take some permutation $\sig$
attaining $\Det{A}$, and let $\al$ be a tangible element of
$\nu$-value slightly greater than $\rone^\nu$. Replacing $a_{\sig
(i),i}$ by $\al \inua{ a_{\sig (i),i}}$ for each $1 \le i \le n$
gives us a matrix close to~$A$ whose determinant is $\al^n \inua
{\Det{A}}\in \tT$, as desired.
\end{proof}

Lemma~\ref{ord1} shows us that much of the matrix theory can be
developed by looking merely at the \quasi-invertible matrices. For
example, Remark~\ref{ord0} could be verified by checking only the
\quasi-invertible matrices. Along these lines, we have:

\begin{prop}\label{ord2}   If $f,g \in \operatorname{CFun} (R^{(n^2)},R^{(m)})$ and $f(A) \lmodg  g(A)$ for all \quasi-invertible matrices $A$, then $f \lmodg g$.
\end{prop}
\begin{proof} Combine Lemma~\ref{ord1} with
Proposition~\ref{dens}~(ii).
\end{proof}

 In view of \cite[Proposition
4.17]{IzhakianRowen2008Matrices}, every quasi-identity matrix $I_A$ is
its own left and right quasi-inverse as well as its own left and
right quasi-identity matrix, and $I_A = {I_A}^\nb$.
In order to obtain the best results, we need to modify our
notion of adjoint.

\begin{rem}\label{lem:porder1} Define $$A^{\Bnb} = A^\nb I_A = A^\nb A
A^\nb.$$
Then $A A^{\Bnb} = A A^\nb I_A  = (I_A)^2 = I_A,$ so $A^{\Bnb}$ is
a
  right quasi-inverse of $A$. By~Remark \ref{lem:porder},
$$  A^{\Bnb} =  A^{\nb} I_A \lmodg A^{\nb} \lmodg A^{\Tnb}
.$$  \end{rem}

In fact, $A^{\Bnb}$ is the ``maximal'' right quasi-inverse of $A$,
in the following sense:

\begin{lem}\label{precad} If $AB = I_A$, then $A^{\Bnb} \lmodg  B.$
\end{lem}
\begin{proof} $A^{\Bnb}= A^\nb I_A=   A^{\nb}(AB)= (A^{\nb} A)B = I'_A B
 \lmodg  B.$
\end{proof}

By symmetry $A^{\Bnb}$ is also the ``maximal'' left quasi-inverse
of $A$ (although the corresponding left and right quasi-identities
$I_A$ and $I'_A$ may differ!) From this point of view, $A^{\Bnb}$
is the ``correct'' supertropical version of the adjoint. (The
distinction between   $A^{\nb}$ and $A^{\Bnb}$ would not arise in
classical matrix algebra.)

The same sort of reasoning  as with Lemma \ref{precad} shows that
$I_A$ is  maximal with respect to the following property:
\begin{rem}\label{precid} $ $
\begin{enumerate} \eroman
    \item  If $A^{\Tnb} B = A^{\Tnb}$, then
$ I_A= A A^{\Tnb}= A A^{\Tnb} B=  I_A B \lmodg B .$ \pSkip
    \item   If $A^{\Bnb} B = A^{\Bnb}$, then $ I_A= A A^{\Bnb} = A
A^{\Bnb} B =  I_A B  \lmodg B .$
\end{enumerate}
\end{rem}

 To proceed further, we
need a result from \cite{IzhakianRowen2008Matrices} that relies
on the Hall Marriage Theorem from graph theory, applied to the digraph of the matrix $A$ (which we recall is the graph whose edges are indexed and weighted by
the entries of $A$).

\begin{lem}\label{rmk:adj01} $|A| \adj{A} \ \nuge \ \adj{A} \, A \, \adj{A}$,
for any matrix $A$.\end{lem}
\begin{proof}
The $(i,j)$-entry of $ \adj{A} A \adj{A} $ is the sum of terms of
the form $a'_{k,i}\, a_{k,\ell} \, a'_{j,\ell},$ each of which we
write out as a product of entries of $A$, thereby corresponding to
a digraph (having multiple edges, each corresponding to one of the
entries in this product) with in-degree $2$ at every vertex except
$j$, and out-degree $2$ at every vertex except $i$. Hence,  by
\cite[Lemma~3.16(iv)]{IzhakianRowen2008Matrices}, we can take out
an $n$-multicycle that has $\nu$-value at most $|A|,$ leaving at
most $a'_{i,j},$ so $|A|\adj{A} \nuge \adj{A} A \adj{A}, $ as
desired.
\end{proof}

\begin{thm}\label{rmk:adj1} For   any quasi-invertible matrix $A$,
 $$A^{\Bnb} \nucong A^\nb \nucong A^{\Tnb}.$$
\end{thm}

\begin{proof} By Lemma
\ref{rmk:adj01}, $A^\nb \ge _\nu A^\nb A A^\nb = A^{\Bnb},$ so we
are done by Remark~\ref{lem:porder1}.
\end{proof}
Recall that the relation $$A^{\Bnb} \lmodg A^\nb \lmodg A^{\Tnb}$$
holds for any matrix $A$.

 As with
\cite {IzhakianRowen2007SuperTropical},
\cite{IzhakianRowen2008Matrices}, we present  our examples in
logarithmic notation (i.e., $-\infty$ is the additive identity
matrix and $0$ is the multiplicative identity matrix); we often
write $-$ for $-\infty$.

\begin{example} In logarithmic notation, for $$ \text{} \  A  =  \left(\begin{matrix} 0& a & - \\
- & 0 & b\\ -  & -  & 0
\end{matrix}\right), \quad \text{ we have } \quad A^\nb = \left(\begin{matrix} 0& a & ab \\
-  & 0 & b\\ -  & -  & 0
\end{matrix}\right),$$ $$I_A =  \left(\begin{matrix} 0& a^\nu & ab^\nu \\
-  & 0 & b^\nu \\ -  & -  & 0
\end{matrix}\right), \quad \text{ and }  \quad A^\Bnb =  \left(\begin{matrix} 0& a^\nu & ab^\nu \\
-  & 0 & b^\nu \\ -  & -  & 0
\end{matrix}\right).$$

\end{example}

\begin{rem} Here is an example where $A^\Bnb$ can be tangible
off the diagonal: In Example \ref{twobytwo}, take $a=d=-\infty,$
and $b,c$ tangible. $I_A = I,$ so $A^\Bnb = A^\nb,$ a tangible
matrix.

(This is the only way of getting such an example. Looking into
the computations of the proof of Lemma~\ref{rmk:adj01}, one sees
that when the determinant of $A$ is attained by a product of terms
including a diagonal entry, then the computation of any
off-diagonal entry of $A^\Bnb $ yields two matching terms
containing~$|A|$, and thus $A^\Bnb$ is ghost off the diagonal.)
\end{rem}

\begin{rem} Although our discussion in this section has focused on nonsingular
matrices, one could   define more generally  $$A^\nb  = \frac
\rone {\widehat {\Det{A}}} \adj{A}$$ whenever $\Det{A}^\nu $ is
invertible in $\tG$. Some computational results are available in this situation, such as $A A^\nb$ being idempotent, but the diagonal is no
longer tangible.\end{rem}

\section{Solving Equations}\label{solveq}

We are ready to turn to one of the main features of this paper.
Our objective in this section is to solve matrix equations over
supertropical domains. We look for tangible
solutions, since any large ghost vector would be a solution. There
is an extensive theory of solving equations over the max-plus
algebra~\cite{ABG}, but the supertropical theory has a different
flavor, relying mostly on standard tools from classical matrix
theory. We work in $R^{(n)},$ with
$\tHz = \tGz^{(n)}.$

In general, although the matrix equation $Ax = v$ need not be
solvable, we shall see in Theorems~\ref{tangsol} and~\ref{tang3}
that $$A x \ \lmodh  \ v, \qquad v = (v_1,\dots, v_n),$$ always has a tangible solution for $x =
(x_1,\dots, x_n)$, and the unique maximal tangible solution can be
computed explicitly, for any $n\times n$ quasi-inventible matrix
$A$ and tangible vector $v \in R^{(n)}$ over a supertropical
domain $R = \RGnu$. (These results are somewhat stronger than
those in \cite[Theorems~6.4 and~6.6]{AGG}, which deal with a
weaker relation.)

\begin{example}\label{counter1} In logarithmic notation, let  $$A =  \left(\begin{matrix} 0 & 10 \\
- & 0
\end{matrix}\right) \qquad \text{ and} \qquad  v = (0,0).$$ We first look for a tangible
solution $x = (x_1, x_2)$ of the equation $Ax + v \in \tGz^{(2)}$,
that is, a tangible solution of the equations $$x_1 + 10x_2 + 0
\in \tGz, \qquad x_2 + 0 \in \tGz,$$ which requires $x_2 = 0$ and
thus $x_1 = 10.$

But this unique tangible solution fails to satisfy the matrix
equation $Ax = v$, which thus has no tangible solutions!
\end{example}

In view of this example, we turn instead to the equation $A x
\lmodh v$, which we solve in its entirety, and obtain a condition
when it gives us a solution to $Ax =v$. (When $v$ is tangible, we
have seen that the equation $A x \lmodh v$ is equivalent to $Ax +v
\in \tHz = \tGz^{(n)}.$) First we dispose of a trivial situation.

\begin{rem}
When $A$ is a singular matrix over a supertropical semifield, then
its rows are tropically dependent, and thus $Ax \in \tHz$
for some tangible vector $x$ by \cite[Theorem
2.10]{IzhakianRowen2009TropicalRank} which could be taken with
$x_k = \Inu{|A_{i,k}|}$ for some $i$ (see \cite[proof of Lemma
2.8]{IzhakianRowen2009TropicalRank}). Accordingly, for any given
vector $v,$ under the mild assumption that $\Inu{|A_{i,k}|} \ne
\rzero$ for each $k,$ the matrix equation $Ax \lmodh v$
has the tangible vector solution $cx$ for any fixed large tangible
constant $c$.
\end{rem}

Here is one case in which we can compute the tangible solution to
$Ax\lmodh (\zero)$. We say that two multicycles are \textbf{disjoint} if they have no common edges.

\begin{prop}\label{sing0} Suppose $\rzero \ne |A| \in \tG,$
and $ |A| = \sum_{\sig \in S_n} a_{1,\sig(1)} \cdots
a_{n,\sig(n)},$ is attained only by tangible terms
$a_{1,\sig(1)} \cdots a_{n,\sig(n)}$ whose corresponding multicycles are disjoint. Let $J = \{ \sig \in S_n: \sig$ attains $|A|\}$; i.e., $\sig \in J$ iff $ a_{1,\sig(1)} \cdots
a_{n,\sig(n)} \nucong |A|.$ Then taking $x$ to be the $i$ column of
${\adj{A}}$, we have $A \inu{x} \in \tHz.$
\end{prop}

\begin{proof} Fix $j$ and write  $|A| = \sum_{j} a_{i,j}a'_{i,j}$ , where $a'_{i,j} =
|A_{i,j}|$ . Note that $x = ({ a'_{i,1}}, \dots, { a'_{i,n}})$, so
the $j$ component of $A\inu{x} $ is $\sum _{k=1}^n a_{j,k}\Inu{
a'_{i,k}}$. By \cite[Remark 4.5]{IzhakianRowen2008Matrices}, this
is ghost unless $i=j$. When $i=j$, we get some value $a = \sum
_{k=1}^n a_{i,k}\Inu{ a'_{i,k}}$,
 which is ghost unless it has a single dominating summand $a_{i,k}\Inu{ a'_{i,k}}$. But $\sum _{k=1}^n a_{i,k}
{a'_{i,k}} = |A|$ is ghost, and is dominated by $a_{i,k} a'_{i,k}$
alone, which thus must be ghost. Since $|A| \ne \rzero,$ we see
that $ a_{i,k}\in \tT,$ implying $a'_{i,k}\in \tG.$

 Taking $$ J_{i,k} = \{ \sig \in J: a_{1,\sig(1)} \cdots a_{i-1,\sig(i-1)} a_{i+1,\sig(i+1)} \cdots a_{n,\sig(n)}: \sig(i) = k\},$$ we have $|A| = a_{i,k} a'_{i,k}$ implying
$$a'_{i.k} = \sum _{\sig \in J_{i,k}} a_{1,\sig(1)} \cdots a_{i-1,\sig(i-1)} a_{i+1,\sig(i+1)} \cdots a_{n,\sig(n)}.$$
By hypothesis, each summand is tangible, so $ J_{i,k}$ has order at least 2.
This shows $J$ has two permutations with the common edge $(i,k),$
contrary to hypothesis.
\end{proof}

\begin{cor}\label{sing} Suppose $\rzero \ne |A| \in \tG,$ but every entry of $A$ and of $\adj{A}$ is in $\tTz$. Then taking $x$ to be the $i$ column of
${\adj{A}}$, we have $A x \in \tHz.$
\end{cor}
\begin{proof} Otherwise, by the contrapositive of the proposition, two permutations $\sig \ne \tau$ attain the determinant where $\sig (i) = \tau (i) = k$ for suitable $i,k,$ and thus $a'_{i.k} \in \tG,$ contrary to hypothesis.
\end{proof}

The same argument will be used in Theorem \ref{sing12} in a more
technical setting, when we consider eigenvalues. Accordingly, we
assume that $A$ is quasi-invertible (which is the same as nonsingular
when $R$ is a supertropical semifield). We start with the tropical
analog of Cramer's rule.

\begin{thm}\label{tangsol} If $A$ is a \quasi-invertible matrix and $v$ is a tangible
vector, then the equation $A x \lmodh  v$ has the tangible vector
solution $x = \Inu{ (A^\nb v)}.$
\end{thm}
\begin{proof}  The proposed solution $x= (x_1,\dots, x_n)$ satisfies $|A| x_k = \hnu \( \sum_j
a'_{j,k}v_j\)$, for $v = (v_1,\dots,v_n)$.
Thus,
\begin{equation}\label{prov2} |A|(Ax)_i = \sum _k \left(
a_{i,k}\, \hnu \bigg ( \sum _j  a'_{j,k} v_j\bigg)\right),
\end{equation} which we want to show ghost-surpasses  $|A|v_i $.
 For $j= i$, we see that $\sum _k a_{i,k} \Inu{(a'_{i,k}v_i)}$
has the same $\nu$-value as $\sum _k a_{i,k}a'_{i,k}v_i,$ which is
$ |A|v_i\in \tTz$, implying
$$\sum
_k a_{i,k}  \Inu{ (a'_{i,k}v_i)} \lmodh  |A|v_i  .$$

Thus, we are done if $\sum _k a_{i,k} \Inu{(a'_{i,k}v_i)}$
dominates $|A|(Ax)_i$, and we may assume that $$|A|(Ax)_i = \sum
_k a_{i,k} \hnu\left( \sum _{j\ne i} a'_{j,k}v_j\right), $$ which
is ghost by Proposition~\ref{tang} (since  $\sum _k a_{i,k} \Inu
{a'_{j,k}} \in \tGz $  by \cite[Remark
4.5]{IzhakianRowen2008Matrices}). Hence, by components,  $
|A|(Ax)\lmodh |A|v  ,$ implying $Ax \lmodh  v  .$
\end{proof}

\begin{note*} Suppose that $A$ is quasi-invertible, and $v \in R^{(n)}$.

\begin{enumerate} \eroman
    \item $ A^{{\Tnb}} v \nucong A^\nb v$, in view of
Theorem~\ref{rmk:adj1}.
    \pSkip

    \item
    When $v= I_A v$, we claim that we have the ``true'' solution
$Ax =v.$ Indeed,  $$v = I_A v = (A A^\nb)v \lmodh A \Inu{( A^\nb
v)} = Ax,$$ so $Ax = v$ since $v$ is  presumed tangible.  \pSkip

    \item From this point of view, the ``good'' vectors for
solving the matrix equation $Ax =v $ are those tangible vectors $v
= I_A w$ for some $w$, since then
$$v = I_A w = I_A^2 w = I_A v.$$ \pSkip
\end{enumerate}

\end{note*}

 Let us turn to the question of uniqueness of our solution. Note that if
$A$ is a nonsingular matrix, then the only tangible solution to $A
x \in \tHz$ is $x = (\rzero),$ in view of \cite[Lemma
6.9]{IzhakianRowen2008Matrices}. On the other hand, we have the
following example.
\begin{example} In logarithmic notation, take  $$A =  \left(\begin{matrix} 5 & 0 \\
5& 1
\end{matrix}\right) \qquad \text{ and} \qquad  v = (5,5).$$
The tangible solution for $Ax \lmodh  v $ obtained from
Theorem~\ref{tangsol} is $$x = \hnu\left(\left(\begin{matrix} -5 & -6 \\
-1 & -1
\end{matrix}\right)v\right) = (0,4),$$ and indeed $Ax = (5,5^\nu)
  \lmodh v$. However,
instead of $x$ we could take $y =(0,\a)$ for every tangible $\a
<_\nu 4$ and get the equality $Ay = v.$

Note that these solutions exist despite the fact that $I_A v \ne
v.$ The supertropical solution is the limiting case of the other
solutions, and would provide the ``maximal'' solution over the
max-plus algebra.
\end{example}

In general, we do have uniqueness in the sense of the following theorem (\ref{tang3}):

\begin{prop} If $Ax  \lmodh  v$ and $Ay \lmodh v$ for tangible vectors
$x$ and $y$, then $A  \Inu{x+y} \lmodh v.$\end{prop}
\begin{proof} This is clear unless some tangible component in $A\Inu{x+y}$,
say the $i$-component, has  $\nu$-value at least that of the
corresponding component $v_i$ in $v$. But then it comes from some
dominating $a_{i,j}x_j$ or  $a_{i,j}y_j$ with $a_{i,j}$ tangible.
Say $a_{i,j}x_j \ge_\nu v_i$ dominates the $i$-component of
$A\Inu{x+y}$. But then   $a_{i,j}x_j  $ is tangible, so either
$a_{i,j}x_j =  v_i$ and we are done, or $a_{i,j}x_j >_\nu v_i$,
and thus by hypothesis $a_{i,j'}x_{j'} = a_{i,j}x_j $ for some
$j'$, implying that the $i$-component of $A\Inu{x+y}$ is
$(a_{i,j}x_j)^\nu$, a ghost, so again we are done.
\end{proof}

It follows that taking the tangible retract of the sum of all
tangible solutions $x$ to $Ax \lmodh v$ gives us the  dominating
tangible solution. Actually, this can be obtained from the
solution given in Theorem~\ref{tangsol}, as we see in the next
result.

\begin{thm}\label{tang3} If $A x \lmodh  v  $ for $A$ quasi-invertible and a tangible vector
$x$, then $x \le_\nu \Inu{(A^{\nb} v)}.$\end{thm}
\begin{proof} First we assume that $A = I_A = (a_{i,j})$
 is a quasi-identity matrix. Since $A^\nb =  I_A^\nb = I_A = A,$ the equation $A x \lmodh v $ has
 the tangible
 solution $y = \Inu{(A^{\nb} v)} =  \Inu{(A v)}$; i.e., for each $i$,  $y_i = \Inu{ a_{i,j}v_j}$
 for suitable $j $ (depending on $i$), and $y_i \nuge a_{i,i}v_i= v_i$.  Note that
 $$Ay \, \nucong \, AA^\nb v \, \nucong \,  I_A v  \, \nucong \,  A ^{\nb} v \, \nucong \,
 y,$$
implying
  $y   =  \Inu{(A y)}$. Thus, $y_i \nuge a_{i,j}y_j$ for
  all $i,j$, and hence, since $a_{i,i} = \rone,$
  $$y_i \ = \  a_{i,i}y_i \ \nuge \ \sum _j a_{i,j}y_j \ \nuge \  v_i.$$

 Suppose \begin{equation}\label{Cram} A x \lmodh v,\end{equation}
 with $x = (x_1, \dots, x_n).$
  We need to show that $y_i \nuge x_i $
 for each $i$.

If not, then, for some $i$, $x_{i}  >_\nu y_{i}$;
 take such an $i_0 = i$ with $\frac{x_{i_0}}{y_{i_0}}$ $\nu$-maximal. (If some $y_i = \rzero,$  we take~${i_0}$ such that  $x_{{i_0}}$ is $\nu$-maximal for which $y_{{i_0}} = \rzero$.)   Since by hypothesis
 $x_{i_0}\in \tTz$, we must have
 $$a_{{i_0},{i_0}}x_{i_0} = x_{i_0}>_\nu y_{i_0} = a_{{i_0},{i_0}} y_{i_0},$$
implying $a_{{i_0},{i_0}}x_{i_0} >_\nu v_{i_0}$, and thus, in view
of \eqref{Cram},  $a_{{i_0},{i_0}}x_{i_0} \le_\nu
a_{{i_0},{i_1}}x_{i_1}$ for some ${i_1}\ne i_0$. Then
\begin{equation}\label{3.3} a_{{i_0},{i_1}}x_{i_1} \,  \nuge \,  a_{{i_0},{i_0}}x_{i_0}
 \nug \, y_{i_0} \, \nuge \, a_{{i_0},{i_1}}y_{i_1}.\end{equation}
Hence, $$\frac{x_{i_1}}{y_{i_1}}  \ \nuge \
\frac{x_{i_0}}{y_{i_0}},$$ so by assumption
$$\frac{x_{i_1}}{y_{i_1}} \  \nucong \
\frac{x_{i_0}}{y_{i_0}},$$ and  the ends of Equation \eqref{3.3}
are $\nu$-matched. Inductively, by the same argument, for each
$t\ge 0$ we get $i_{t+1}\ne i_t$ such that $y_{i_t} \nucong
a_{i_t, i_{t+1}} y_{i_{t+1}},$ and we consider the path obtained
from the indices $i_0, i_1, \dots, i_t$  in the reduced digraph of
$A$ (cf.~\cite[Section 3.2]{IzhakianRowen2008Matrices}). For $t>n$
this must contain   a cycle, so there are $s<t$ such that
$$y_{i_s} \, \nucong \, y_{i_s}  a_{i_s,i_{s+1}}\cdots a_{i_t,i_{t+1}}
$$ Hence, $a_{i_s,i_{s+1}}\cdots a_{i_t,i_{t+1}} \nucong  \rone$,
contradicting the fact that $A$ is a quasi-identity matrix (and
thus cannot have a loopless cycle of  weight $\nucong \rone$).

In general, suppose that $A x\lmodh v$. Then $I'_A x =   A^{\nb}A x \lmodh
A^{\nb} v,$ implying by the previous case that
$$x \ \nule \ \hnu({I'_A}^{\nb} A^{\nb} v) \  =  \ \hnu((I'_A A^{\nb}) v) \ = \ \hnu((A^{\Bnb}) v) \
= \ \hnu(A^{\nb} v) ,$$ in view of Theorem~\ref{rmk:adj1}.
\end{proof}

This theorem does not provide tangible solutions when $v = \zero,$
i.e., $A x \in \tHz$ for $A$ quasi-invertible, since then
$\hnu{(A^{\Tnb} v)} = \zero$ and we have no nontrivial solutions;
in this sense, Proposition~ \ref{sing} is sharp.

\section{Properties of the adjoint and tangible adjoint}

 \begin{example}\label{triang} Let us compute $A^{\nb \nb} = {(A^{\nb})}^\nb$ for the triangular
 nonsingular
 matrix

$$A = \vvMat{a_{1,1}}{a_{1,2}}{a_{1,3}}
            {-}{a_{2,2}}{a_{2,3}}
            {-}{-}{a_{3,3}},
$$
Then  $|A| = a_{1,1}\, a_{2,2} \, a_{3,3}$ and

$$A^\nb = \frac{\rone}{|A|} \vvMat{a_{2,2} a_{3,3}}{a_{1,2}a_{3,3}}{a_{1,2}a_{2,3}+a_{1,3}a_{2,2} }
            {-}{a_{1,1}a_{3,3}}{a_{1,1}a_{2,3}}
            {-}{-}{a_{1,1}a_{2,2}},  \quad \text{so} \quad |A^\nb| = \frac{ \rone }{ |A|},$$
and

$$ A^{\nb \nb} =  \frac{\rone}{|A^{\nb}|}\adj{A^{\nb}} = {|A|}\adj{A^{\nb}} =
\frac{\rone}{|A|}  \vvMat{a_{1,1}|A|}{
a_{1,1}a_{1,2}{a_{2,3}a_{3,3}}^\nu +a_{1,2}|A|}{a_{1,3}|A|}
            {-}{a_{2,2}|A|}{a_{2,3}|A|}
            {-}{-}{a_{3,3}|A|}.$$
Clearly $a_{1,1}a_{1,2}a_{2,3}{a_{3,3}}^\nu +a_{1,2}|A| \lmodg
a_{1,2}|A|$, and thus $A^{\nb \nb} \lmodg  A$.

For further reference, we note that $$A^\Bnb = A^\nb I_A =
\frac{\rone}{|A|} \vvMat{a_{2,2}
a_{3,3}}{a_{1,2}{a_{3,3}}^\nu}{a_{1,2}{a_{2,3}}^\nu+a_{1,3}{a_{2,2}}^\nu
}
            {-}{a_{1,1}a_{3,3}}{a_{1,1}{a_{2,3}}^\nu}
            {-}{-}{a_{1,1}a_{2,2}},  \quad \text{with } \quad |A^\nb| = \frac{ \rone }{ |A|}.$$

            Consequently, $$A^\Bnbnb =  \frac{\rone}{|A^{\nb}|}\adj{A^{\nb}} =
            \frac{\rone}{|A|}  \vvMat{a_{1,1}|A|}{
a_{1,1}a_{1,2}a_{2,3}{a_{3,3}}^\nu
+a_{1,2}|A|^\nu}{a_{1,3}|A|^\nu}
            {-}{a_{2,2}|A|}{a_{2,3}|A|^\nu}
            {-}{-}{a_{3,3}|A|},$$
            which is not necessarily $A^\nbnb$ (although they are
            $\nu$-matched).
\end{example}

\begin{rem} Although $A^{\nb \nb} \ne A$ in general, one does get $A^{\nb \nb} \lmodg
A$, as a consequence of Akian, Gaubert, and Guterman
\cite[Theorem~4.21]{AGG}, quoted above as
Theorem~\ref{STP1}.\end{rem}

 Here are some more computations with
adjoints.

\begin{thm}\label{doublead}
 $   \adj{A} \adj{\adj{A}}\adj{  A} \nucong |A|^{n-1}\adj{ A}$ for any $n \times n$ matrix $A$.
\end{thm}
\begin{proof} Another application of Hall's Marriage Theorem.  Let $\adj{\adj{A}} = (a''_{i,j})$.
Clearly
$$   \adj{A} \adj{\adj{A}}\adj{  A}\nuge |A|^{n-1}\adj{ A},$$ by \cite[Theorems 4.9(ii) and~4.12]{IzhakianRowen2008Matrices}, so
it suffices to prove $\nule$. But the $(i,j)$ entry of the left
side is a sum of elements of the form $a'_{i,k} a''_{\ell,k}
a'_{\ell,j}$ which has in-degree $n$ in all indices except $i$
(which has in-degree $n-1$), and out-degree~$n$ in all indices
except $j$ (which has out-degree $n-1$), and thus by
\cite[Lemma~3.16(iv)]{IzhakianRowen2008Matrices} we can factor out
$(n-1)$ \nmulti s, each of weight  $\le_\nu |A|,$ and conclude
that each summand $\nule |A|^{n-1} a'_{i,j}.$
\end{proof}

  \begin{cor}\label{doublead1} If $A$ is a quasi-invertible matrix,
then   $A^{\nb}A^{\nb \nb} A^{\nb} \nucong A^{\nb}.$
\end{cor}

We are finally ready for the connection between left
quasi-identities and right quasi-identities; the key is to switch
from $A$ to $A^{\Tnb}.$

\begin{cor}\label{switch} If $A$ is a quasi-invertible matrix, then
 $    A^{\Tnb \Tnb} A^{\Tnb} =  I_A .$ In other words, $I_A =
 I'_{A^{\Tnb}}=  I'_{A^{\Bnb}}.$
\end{cor}
\begin{proof} $   I'_{A^{\Tnb}} =  A^{\Tnb \Tnb} A^{\Tnb} \nule  I_A $ by Corollary~ \ref{doublead1} and Remark \ref{precid}, but $    A^{\Tnb \Tnb}
A^{\Tnb} \nuge    A  A^{\Tnb} = I_A$ by Theorem~ \ref{rmk:adj}.
Hence the entries of $I_A$ and $ I'_{A^{\Tnb}}$ have the same
respective $\nu$-values. We conclude by noting that both $I_A$ and
$ I'_{A^{\Tnb}}$ are tangible on the diagonal and ghost off the
diagonal. ($ I'_{A^{\Tnb}}=  I'_{A^{\Bnb}}$ by
Remark~\ref{match1}.)
\end{proof}
\begin{cor}\label{bysym} By symmetry, $I'_A =
 I_{A^{\Tnb}}= I_{A^{\Bnb}}.$
\end{cor}

At  last we have resolved the enigma arising from Example~\ref{twobytwo}: The left quasi-identity  of
a matrix corresponds to the right quasi-identity of its adjoint, and vice versa.

\section{Application: Supertropical eigenvectors}

Recall from \cite{IzhakianRowen2008Matrices} that a tangible
vector $v$ is a \textbf{supertropical eigenvector} of $A$, with
\textbf{supertropical eigenvalue} $\bt \in \tTz$,
 if $$ A v \, \lmodg  \, \bt v,  $$
i.e., if  $A v = \bt v + \text{ghost}$.

In \cite[Theorem 7.10]{IzhakianRowen2008Matrices} we showed that
every root of the characteristic polynomial of $A$ is a
supertropical eigenvalue. However, the proof does not give much
insight into the specific eigenvector. Here, we use the properties
of the adjoint matrix to compute explicitly the supertropical
eigenvectors; this method is expected to be a useful tool for developing
linear algebra.

Recall the following observation from \cite[Remark
7.9]{IzhakianRowen2008Matrices}:
\begin{rem} If $\widehat A$ is a tangible matrix (i.e., all entries
are in $\tTz$), such that ${\widehat A}  \nucong  A ,$ then
every tangible supertropical eigenvector of $\widehat A$ is also a
supertropical eigenvector of $A$ with respect to the same
supertropical eigenvalue. \end{rem} In view of this remark, in the
sequel,  we may assume that all of the entries of our matrix $A$
are tangible.

\begin{defn} A polynomial is \textbf{quasi-tangible} if all of its
coefficients except perhaps the constant term are
tangible.\end{defn}

 {\it We also assume from now on that the
essential part ${f_A}^{\essn}$, cf. \cite[Definition
4.9]{IzhakianRowen2007SuperTropical}, of the characteristic
polynomial $f_A$ is quasi-tangible.}

(The reason that we exclude the constant term from our hypothesis
is that we want to permit $\rzero$ to be an eigenvalue.) We write
$\bt_1, \dots, \bt_t$ for the distinct roots of ${f_A}^{\essn} $,
written in  order of descending $\nu$-values. Thus,  $\bt_\ell \in
\tTz$ for each $\ell \le t,$ with $\bt_\ell \in \tT$ for each
$\ell < t.$   Recall from \cite [Theorem 7.10]
{IzhakianRowen2008Matrices} that
\begin{equation}\label{eq:esspoly} {f_A}^{\essn} = \la ^n + \sum
_{\ell =1}^{t}  \a_{\ell} \la ^{n-m_\ell} ,
\end{equation}
where $\a_{\ell}$ equals the maximal weight (with respect to
$\nu$-value) of an $m_\ell$-multicycle in the digraph of $A$,
which we denote as ~$\cyc_{\ell}$.
\begin{rem} $\cyc_{\ell}$ is unique for each $\ell < t,$ since
$\a_{\ell}$ is assumed tangible.  \end{rem}

Since $\beta_\ell$ is a tangible root of ${f_A}^{\essn}$, we have
\begin{equation}\label{btal} \beta_\ell^{n- m_{\ell-1}} \a_{
\ell-1} = \beta_\ell^{n- m_{\ell}}\a_{\ell},\end{equation}
implying
\begin{equation}\label{btal1} \beta_\ell^{ m_{\ell}-m_{\ell-1}} \a_{
\ell-1} = \a_{\ell}.\end{equation} (For $\ell = t,$ we only have
this up to $\nu$-values, and only when $\beta_t \ne \rzero$.)
Hence, $\bt_\ell^\nu$ equals the negative of the  slope of the
edge connecting $(m_\ell, \a_{\ell}^\nu)$ to $( m_{\ell
-1},\al_{\ell-1}^\nu)$ in the graph of the coefficients of $f_A$.

 Here is an intuitive way of computing a
supertropical eigenvector. Let $$B_\ell = A + \bt _\ell I.$$ In
\cite[Proposition 7.8]{IzhakianRowen2008Matrices}, we showed that
$\Bl$ is a singular matrix for every tangible root $\bt_\ell$ of
${ f}_A$. Taking an arbitrary vector $w$ and letting $v =
\adj{B_\ell}w,$ we have
$$Av + \bt_\ell v = (A + \bt_\ell I)(\adj{B_\ell} w) = B_\ell \adj{B_\ell} w $$ is ghost.
If $w$ can be chosen such that $v$ is tangible, this implies by
\cite[Lemma~ 7.4]{IzhakianRowen2008Matrices} that $v$ is a
supertropical eigenvector. This is the motivation for the next
result. First we make our discussion more explicit.

\begin{remark}\label{rmk:1} Write $ \Bl = (b_{i,j})$. Thus, $b_{i,j} = a_{i,j}$ for $i\ne j$,
and $b_{i,i} = a_{i,i} + \bt_\ell$.
$$ |\Bl | = ( \al_\ell\bt_\ell^{n-m_\ell})^\nu.$$
The determinant of $\Bl = A + \bt_\ell I$ comes from  the \nmulti
s of maximal weight.  Since $\bt_\ell$ is a tangible root of ${
f}_A$, there are two dominating contributions: One comes from
$n-m_\ell$ entries of $\bt_\ell$ along the diagonal, where the
remaining $m_\ell$ entries must come from the dominating
$m_\ell$-multicycle $\cyc_{\ell}$ in the digraph of $A$. (Note
that for $\ell=t$ this contribution might not be unique.) The
other dominating   term comes from $n-m_{\ell-1}$ entries of
$\bt_\ell$ along the diagonal, where the remaining $m_{\ell-1}$
entries must come from the dominating $m_{\ell-1}$-multicycle
$\cyc_{\ell-1}$ (in the digraph of $A$), and we also have
$$ |\Bl | = ( \al_{\ell-1}\bt_{\ell}^{n-m_{\ell-1}})^\nu$$
(which follows from Equation \eqref{btal}).

Formally take $\a _0 = \rone.$ Applying induction to \eqref{btal}
yields \begin{equation}\label{detrule} \a_{\ell} = \prod _{u
=1}^\ell \beta_u^{ m_{u}-m_{u-1}} ,\end{equation} and thus
\begin{equation}\label{detrule2} |\Bl| = (\beta_\ell^{n- m_{\ell-1}} \prod
_{u =1}^{\ell-1} \beta_u^{ m_{u}-m_{u-1}})^\nu.\end{equation}

\end{remark}

We introduce some more notation: For any root $\bt_\ell$ of ${f_A}
^{\essn}
  $, let \begin{equation}\label{eq:J}
J_\ell = \{ \text{Vertices of } \cyc_\ell\}\setminus \{
\text{Vertices of } \cyc_{\ell -1}\}.
\end{equation} (Note that this definition
is well-defined even for $\ell =t,$ since every \nmulti\ contains
all the vertices $\{1, \dots, n\}$ in the digraph of $A$.) Write
$b'_{i,j}$ for the $(i,j)$ minor of $B_\ell = A+ \bt_{\ell} I$.

\begin{lem}\label{sing11}  $a_{i,i} \le _\nu
\bt _\ell$  for any $i \in J_\ell,$ and thus $$b'_{i,i} =
\al_{\ell-1}\bt_{\ell}^{n-m_{\ell-1}-1},$$ which is tangible and
has the same $\nu$-value as $\frac{|B_\ell|}{\bt_\ell}.$
\end{lem}
\begin{proof} By definition, $\cyc_{\ell-1}$ occurs in the digraph
of the minor $A_{i,i}$, of weight $\a_{\ell-1},$ so $\cyc_{\ell-1}
\cup \{a_{i,i}\}$ is an $m_{\ell-1} +1$ multicycle of weight
$\a_{\ell-1}a_{i,i},$ and the coefficient of $\la ^{n - (m_{\ell-1}
+1)}$ in $f_A$ must have at least its $\nu$-value. If $a_{i,i} >
_\nu \bt _\ell,$ then $\cyc_{\ell-1}\cup \{a_{i,i}\}$ would produce the single
dominant value for $f_A(\bt_\ell),$ contrary to hypothesis.

It follows that $b_{i,i} = a_{i,i} + \bt_\ell \nucong \bt_\ell.$
Remark~\ref{rmk:1} then implies
$$  |\Bl | \, \nucong \,    \al_{\ell-1}\bt_{\ell}^{n-m_{\ell-1}}
\, \nucong \,  b'_{i,i}b_{i,i} ,$$ and we conclude that $b'_{i,i}
\nucong \frac{|B_\ell|}{\bt_\ell}$. Furthermore, the only terms
which can contribute to $|B_\ell|$ are $
\al_\ell\bt_\ell^{n-m_\ell}$ and
$\al_{\ell-1}{\bt_{\ell}^{n-m_{\ell-1}}}.$ But, by choice of $i$,
$a_{i,i}$ cannot occur in $\cyc_{\ell-1}$. Hence, the only
contribution to $b'_{i,i}$ is
$\al_{\ell-1}\bt_{\ell}^{n-m_{\ell-1}-1}$, as desired.
\end{proof}

\begin{thm}\label{sing12}
  For any root $\bt_\ell$ of   ${f_A}
  $,
 and  for any $i \in J_\ell,$ taking
$v$ to be the $i$ column of $\adj{B_\ell}$, we have $A\hat v
\lmodg \bt_\ell \hat v$. (In other words, $\hat v$ is a
supertropical eigenvector of $A$.)
\end{thm}
\begin{proof} In view of \cite[Lemma~
7.4]{IzhakianRowen2008Matrices}, it suffices to prove that $A\hat
v + \bt \hat v \in \tHz;$ i.e., that $B_\ell\hat v\in \tHz.$ Write
$b'_{i,j}$ for the $(i,j)$ minor of $B_\ell$. By definition,
  $$\hat v = (\Inu{ b'_{i,1}}, \dots, \Inu{ b'_{i,n}}).$$ In view of
Proposition~\ref{tang} and
\cite[Remark~4.5]{IzhakianRowen2008Matrices}, the $j$ component of
$B_\ell\hat v$ is  ghost unless $i=j$, and it suffices to prove
that $\sum _{k=1}^n b_{i,k}\Inu{ b'_{i,k}}$ is ghost. This is
clear unless the right side has a single dominating summand
$b_{i,k}\Inu{ b'_{i,k}}$. But $\sum _{k=1}^n b_{i,k} b'_{i,k} =
|A|$ is ghost, and is dominated by $b_{i,k} b'_{i,k}$ alone, which thus
must be ghost. Furthermore, by Remark~\ref{rmk:1},
$$b_{i,k} b'_{i,k} \nucong \al_\ell\bt_\ell^{n-m_\ell} \nucong
\al_{\ell-1}\bt_{\ell-1}^{n-m_{\ell-1}};$$ in other words, the two
terms on the right side must occur in $b_{i,k} b'_{i,k}$ as the dominating terms. In
particular, one summand of $b_{i,k} b'_{i,k}$ must consist of
diagonal elements $\bt_\ell$ and the \multi\ $\cyc_{\ell-1}$. But,
by choice of~$i$, $a_{i,k}$~cannot occur in $\cyc_{\ell-1}$; if
$k\ne i$ then $b_{i,k} = a_{i,k}$ cannot occur in this summand, a~ contradiction.

Thus, $k = i$ and Lemma \ref{sing11} shows that $a_{i,i}$ is part
of $\cyc_\ell,$ implying that  $$\cyc_\ell = \cyc_{\ell-1} \cup \{
a_{i,i}\}.$$ But then $\bt_\ell \nucong a_{i,i} ,$ implying
$b_{i,i} = \bt_\ell + a_{i,i} \in \tG,$ and thus $b_{i,i} \Inu{
b'_{i,i}} \in \tG,$  as desired.
\end{proof}

Here is a surprising example.

\begin{example} A matrix $A$ whose characteristic polynomial has distinct roots,
but the supertropical eigenvectors are supertropically dependent.
Let

\begin{equation} A =\left(\begin{matrix} 10 & 10 & 9 & - \\  9 & 1 & -  & - \\-  & - & -  & 9 \\  9 & -  &-  & -  \end{matrix}\right).
\end{equation}
Notation as in Remark~\ref{rmk:1}, \begin{itemize} \item $\cyc_1 =
(1),$ of weight 10, \pSkip  \item $\cyc_2 = (1 ,2),$ of
weight 19, \pSkip \item  $\cyc_3 = (1,3,4 ),$ of weight 27,
\pSkip \item $\cyc_4 = (1,3,4 )(2),$ of weight 28. \end{itemize}

Hence, the characteristic polynomial of $A$ is
$$f_A = \la^4 + 10\la^3 + 19 \la^2 + 27 \la + 28,$$
whose roots are $10,9,8,1$, which are the respective eigenvalues $\beta_1, \bt_2, \bt_3, \bt _4$ of $A$.

We also have $J_1 = \{1\},$ $J_2 = \{2\},$ $J_3 = \{3,4\},$ and
$J_4 = \{2\}.$ (The pathology of this example is explained by the
fact that $J_4 = J_2,$ cf. \eqref{eq:J}).

For each $1 \le \ell \le 4,$ let us compute $B_\ell = A + \bt_\ell I$ and
$\widehat {v_\ell}$, where $v_\ell$ is the column of $\adj {B_\ell}$ corresponding to the $j$ column with $j \in J_\ell.$

\begin{description}
    \item[$\bt_1 = 10$]  $ $ \\
 $B_1 = A + \bt_1 I = \left(\begin{matrix} 10^\nu & 10 & 9 & - \\  9 & 10 & -  & - \\-  & - & 10 & 9 \\  9 & -  &-  & 10\end{matrix}\right),$
 so $\widehat {v_1}=
\(\begin{array}{c}
  30 \\
  29 \\
  28 \\
  29 \\
\end{array}\),$ the first column of $\widehat {\adj {B_1}}$, and  $A \widehat {v_1}= 10  \widehat {v_1}$.\pSkip

\item[$\bt_2 = 9$]  $ $ \\
$B_2 = A + \bt_2 I = \left(\begin{matrix} 10 & 10 & 9 & - \\  9 &
9 & -  & - \\-  & - & 9 & 9 \\  9 & -  &-  & 9
\end{matrix}\right),$ so $ \widehat {v_2}=
\(\begin{array}{c}
  28 \\
  28 \\
  28 \\
  28 \\
\end{array}\),$ the second column
of $\widehat {\adj {B_2}}$, and $A \widehat {v_2}=
\(\begin{array}{c}
  38^\nu \\
  37^{\phantom{\nu}} \\
  37^{\phantom{\nu}} \\
  37^{\phantom{\nu}} \\
\end{array}\) \lmodg 9 \widehat {v_2}$.\pSkip

\item[$\bt_3 = 8$]  $ $ \\
$B_3 = A + \bt_3 I = \left(\begin{matrix} 10 & 10 & 9 & - \\  9 &
8 & -  & - \\-  & - & 8 & 9 \\  9 & -  &-  & 8
\end{matrix}\right),$ so  $v_3=
\(\begin{array}{c}
  25 \\
  26 \\
  27 \\
  26 \\
\end{array}\),$  the third column of
$\widehat {\adj {B_3}}$. (Or we could use the fourth column, which is $
\(\begin{array}{c}
  26 \\
  27 \\
  28 \\
  27 \\
\end{array}\)= 1\cdot v_3,$ so we would obtain the same result.)
 $A \widehat {v_3}=
\(\begin{array}{c}
  36^\nu \\
  34^{\phantom{\nu}}\\
  35^{\phantom{\nu}} \\
  34^{\phantom{\nu}} \\
\end{array}\) \lmodg 8 \widehat {v_3}$.
\pSkip

\item[$\bt_4 = 1$]  $ $ \\
$B_4 = A + \bt_4 I = \left(\begin{matrix} 10 & 10 & 9 & - \\  9 &
1 & -  & - \\-  & - & 1 & 9 \\  9 & -  &-  & 1^\nu
\end{matrix}\right),$ so
 $v_4=
\(\begin{array}{c}
  12 \\
  27 \\
  28 \\
  20 \\
\end{array}\),$ the {\it second}
column of $\widehat {\adj {B_4}}$, and
 $A \widehat {v_4}=
\(\begin{array}{c}
  37^\nu \\
  28^{\phantom{\nu}} \\
  29^\nu \\
  21^\nu \\
\end{array}\) \lmodg 1 \widehat {v_4}$. \pSkip
\end{description}

Combining these four column vectors yields the matrix
$$V = \left(\begin{matrix} 30 & 28 & 25 & 12\\ 29 & 28 & 26 & 27\\ 28 & 28 & 27 & 28\\ 29 & 28 & 26 & 20 \end{matrix}\right),$$
which is singular, having determinant $112^\nu = v_{1,1} v_{2,4} v_{3,3} v_{4,2} =  v_{1,1} v_{2,2} v_{3,4} v_{4,3}.$
\end{example}


\begin{thebibliography}{1}

\bibitem{ABG}
M.~Akian, R.~Bapat, and S.~Gaubert.
\newblock Max-plus algebra, 2008.
\newblock Preprint.

\bibitem{AGG}
M.~Akian,   S.~Gaubert, and A.~Guterman.
\newblock Linear independence over tropical semirings and beyond.
\newblock Contemp. Math., to appear (Preprint at arXiv:math.AC/0812.3496v1).


\bibitem{zur05TropicalAlgebra}
Z.~Izhakian.
\newblock Tropical arithmetic and algebra of tropical matrices.
\newblock {\em Communincation in Algebra},  37:{1-24}, {2009}.
\newblock (preprint at arXiv:math.AG/0505458).


\bibitem{zur05TropicalRank}
Z.~Izhakian.
\newblock The tropical rank of a tropical matrix.
\newblock Preprint at arXiv:math.AC/0604208.


\bibitem{IzhakianRowen2007SuperTropical}
Z.~Izhakian and L.~Rowen.
\newblock Supertropical algebra.
\newblock Preprint at arXiv:0806.1175.

\bibitem{IzhakianRowen2008Matrices}
Z.~Izhakian and L.~Rowen.
\newblock Supertropical matrix algebra.
\newblock {\em Israel J. Math.}, to appear (Preprint at arXiv:0806.1178).

\bibitem{IzhakianRowen2009TropicalRank}
Z.~Izhakian and L.~Rowen.
\newblock The tropical rank of a tropical matrix.
\newblock {\em Communications in Algebra}, to appear.


\end{thebibliography}

\end{document}